\documentclass[11pt]{amsart}
\usepackage{amsmath, amsthm, amsfonts, amssymb}
\usepackage{mathtools}
\usepackage{enumitem}

\usepackage{tikz}
\usepackage{tikz-cd}
\usepackage[all]{xy}

\usepackage[
urlcolor=blue, linktocpage, hyperindex=true, colorlinks=true,
linkcolor=blue, citecolor=blue,
]{hyperref}

\usepackage{color}
\theoremstyle{plain}
\newtheorem{thm}{Theorem}[section]

\newtheorem{conjecture}[thm]{Conjecture}
\newtheorem{conj}[thm]{Conjecture}
\newtheorem{cor}[thm]{Corollary}
\newtheorem{lem}[thm]{Lemma}
\newtheorem{lemma}[thm]{Lemma}
\newtheorem{prop}[thm]{Proposition}

\theoremstyle{definition}

\newtheorem{rmk}[thm]{Remark}

 \usepackage{todonotes}

\newcommand{\alb}{\mathrm{alb}}
\newcommand{\Alb}{\mathrm{Alb}}
\newcommand{\alg}{\mathrm{alg}}
\newcommand{\CH}{\mathrm{CH}}
\newcommand{\hhom}{\mathrm{hom}}
\newcommand{\id}{\mathrm{id}}
\newcommand{\im}{\mathrm{Im}}

\newcommand{\Pic}{\mathrm{Pic}}
\newcommand{\rat}{\mathrm{rat}}
\newcommand{\supp}{\mathrm{supp}}
\newcommand{\Sym}{\mathrm{Sym}}
\newcommand{\tr}{\mathrm{tr}}

\newcommand{\mo}{\mathcal{O}}

\newcommand{\Aut}{\mathrm{Aut}}

\newcommand{\CC}{\mathbb{C}}
\newcommand{\QQ}{\mathbb{Q}}

\newcommand{\ZZ}{\mathbb{Z}}
\newcommand{\PP}{\mathbb{P}}

\author{Jiabin Du}
\author{Wenfei Liu}
\title{On symplectic automorphisms of elliptic surfaces acting on $\CH_0$}
	\date{\today}
	\address{Xiamen University  \\ School of Mathematical Sciences \\ Siming South Road 422 \\ Xiamen, Fujian 361005 (China)}
\email{jiabindu@stu.xmu.edu.cn}
\email{wliu@xmu.edu.cn}
\begin{document}
	\begin{abstract}
Let $S$ be a complex smooth projective surface of Kodaira dimension one. We show that the group $\Aut_s(S)$ of symplectic automorphisms acts trivially on the Albanese kernel $\CH_0(S)_\alb$ of the $0$-th Chow group $\CH_0(S)$, unless possibly if the geometric genus and  the irregularity satisfy $p_g(S)=q(S)\in\{1,2\}$. In the exceptional cases, the image of the homomorphism $\Aut_s(S)\rightarrow \Aut(\CH_0(S)_\alb)$ has order at most 3.
	
	Our arguments actually take care of the group $\Aut_f(S)$ of fibration-preserving automorphisms of elliptic surfaces $f\colon S\rightarrow B$. We prove that, if $\sigma\in\Aut_f(S)$ induces the trivial action on $H^{i,0}(S)$ for $i>0$, then it induces the trivial action on $\CH_0(S)_\alb$. As a by-product we obtain that if $S$ is an elliptic K3 surface, then $\Aut_f(S)\cap \Aut_s(S)$ acts trivially on $\CH_0(S)_\alb$.
\end{abstract}
	\maketitle

\section{Introduction}
We work over the complex numbers $\CC$ in this paper. 

For a complex smooth projective variety $X$ of dimension $d$, its $k$-th Chow group $\CH_k(X)$ is defined to be $Z_k(X)/\sim_\rat$, where $Z_k(X)$ is the free abelian group on the $k$-dimensional closed subvarieties of $X$ and $\sim_\rat$ denotes the rational equivalence. We have $\CH_d(X)=\ZZ[X]\cong\ZZ$ and $\CH_{d-1}(X) = \Pic(X)$. However, $\CH_{k}(X)$ becomes very hard to compute for $0\leq k\leq d-2$. For example, for a smooth projective surface $S$ with $p_g(S)>0$, the degree zero part $\CH_0(S)_{\hhom}$ of $\CH_0(S)$ is infinite dimensional \cite{Mum69}. \footnote{This means that  the natural map $\Sym^n(S)\times \Sym^n(S)\to \CH_0(S)_\hhom, (A,B)\mapsto [A-B]$ is not surjective for any natural number $n$, where $\Sym^n(S)$ denotes the $n$-th symmetric product of $S$.}

The Bloch--Beilinson conjecture predicts the existence of a finite decreasing filtration on each Chow group $\CH_k(X)_\QQ$ with rational coefficients whose graded pieces are, in terms of correspondence between smooth projective varieties, controlled by the Hodge decomposition of the cohomology groups; see \cite[Section 11.2.2]{Voi03} for precise statements. The philosophy behind this conjecture is that the topology (Hodge theory) determines the algebraic geometry of  cycles.  

Specifically for the $0$-th Chow group, one defines the kernel of the degree map
\[
\CH_0(X)_\hhom:= \ker  (\CH_0(X)\xrightarrow{\deg} \ZZ)
\]
and in turn the kernel of the Albanese map
\[
\CH_0(X)_\alb:= \ker  (\CH_0(X)_\hhom\xrightarrow{\alb} \Alb(X))
\]
Then $ \CH_0(X)\supset\CH_0(X)_\hhom\supset\CH_0(X)_\alb$ are expected to be the first three terms of the Bloch--Beilinson filtration for $\CH_0(X)$, and if $X=S$ is a surface, then this should be the full filtration. As a consequence, one expects the following
\begin{conjecture}\label{conj: sym cor}
	Let $S$ a smooth projective surface. Let $X$ be a smooth projective variety,  and $\Gamma\in \CH^2(X\times S)$ a cycle of codimension $2$. If  $[\Gamma]^{\ast}\colon H^{2,0}(S)\to H^{2,0}(X)$ vanishes, then $\Gamma_{\ast}: \CH_0(X)_{\alb}\to \CH_0(S)_{\alb}$ also vanishes.
\end{conjecture}	
Taking $X=S$ and $\Gamma=\Gamma_\sigma-\Delta_S$, where $\Gamma_{\sigma}$ is the graph of an automorphism $\sigma\in \Aut(S)$ and $\Delta_S\subset S\times S$ is the diagonal, we obtain the following special and more tractable case of Conjecture~\ref{conj: sym cor}:
\begin{conj}\label{conj: sym} 
	Let $S$ be a smooth projective surface. Then any symplectic automorphism acts trivially on  $\CH_0(S)_{\alb}$.
\end{conj}
Here an automorphism $\sigma\in\Aut(S)$ is called \emph{symplectic} if the induced map $\sigma^*: H^{2,0}(S)\rightarrow H^{2,0}(S)$ is the identity. We will use $\Aut_s(S)$ to denote the group of symplectic automorphisms of $S$. 

Surfaces with $\kappa(S)\leq 1$ and $p_g(S)=0$ have trivial $\CH_0(S)_\alb$ by \cite{BKL76}, so Conjecture~\ref{conj: sym} is automatically true in this case. The Bloch conjecture (\cite[Conjecture~1.8 and Proposition~1.11]{Blo10}), which is again a consequence of Conjecture~\ref{conj: sym cor} (take $\Gamma=\Delta_S$), asserts that surfaces of general type with $p_g(S)=0$ also have trivial $\CH_0(S)_\alb$. This has been verified in some special cases by various authors: surfaces with ``enough automorphisms" such as Godeaux surfaces, Burniat-Inoue surfaces, Campedelli surfaces, and their alike \cite{IM79, Bar85,Che13, BC13, Bau14, BF15}, surfaces with ``finite dimensional Chow motives'' such as surfaces rationally dominated by a product of curves \cite{Kim05}, surfaces with ``nice moduli spaces'' such as Catanese surfaces, Barlow surfaces \cite{Voi14} and some numerical Campedelli surfaces \cite{Lat20}, and so on. Since the complete classification for surfaces of general type with $p_g=0$ is still unknown and the effective methods avoiding the classification results of such surfaces are not established, the Bloch conjecture is open by now (\cite{BCP11}).

For surfaces with $p_g(S)>0$, the Albanese kernel $\CH_0(S)_\alb$ is huge (\cite{Mum69}). Nevertheless, Conjecture~\ref{conj: sym} has been confirmed for abelian surfaces \cite{BKL76, Paw19} as well as symplectic automorphisms of finite order of K3 surfaces \cite{Voi12, Huy12}. For Kummer K3 surfaces, infinite order symplectic automorphisms coming from the covering abelian surface are treated in \cite{Paw19}.

The main result of this paper is
\begin{thm}\label{thm: main1}
	Conjecture~\ref{conj: sym} holds for surfaces $S$ with Kodaira dimension $\kappa(S)=1$, unless possibly $q(S)=p_g(S)\in\{1,2\}$. In these cases, the image of the homomorphism $\Aut_s(S)\rightarrow \Aut(\CH_0(S)_\alb)$ has order at most $3$.
\end{thm}

Strengthening the condition of Conjecture~\ref{conj: sym} (compare \cite[Conjecture 1.1]{Voi12}), we have 
\begin{thm}\label{thm: main2}
	Let $S$ be a smooth projective surface with $\kappa(S)=1$. If an automorphism $\sigma\in \Aut(S)$ induces the trivial action on $H^{i,0}(S)$ for $i>0$, then it induces trivial action on $\CH_0(S)_\alb$.
\end{thm}

The results are based on the following elementary observation about the zero cycles of an elliptic surface $f\colon S\rightarrow B$: the Albanese kernel $\CH_0(S)_\alb$ is contained in the so-called \emph{$f$-kernel}
\[
\CH_0(S)_f:=\ker(\CH_0(S)\xrightarrow{f_*} \CH_0(B)).
\]
For any cycle class $\alpha\in \CH_0(S)_f$, one can find a positive integer $d$ and finitely many $\alpha_i\in \CH_0(S)_\hhom$ such that 
\begin{equation}\label{eq: dec}
	d\alpha = \sum \alpha_i\in \CH_0(S),
\end{equation}
and the support $\supp(\alpha_i)$ lies on a single fiber for each $i$ (Lemma~\ref{lem: fib cycle}).

Let $\Aut_f(S)$ be the subgroup of automorphisms of $S$ preserving the fibration structure $f$; for the precise definition, see Section $2$. What we are dealing with in this paper is in fact the group $\Aut_f(S)\cap \Aut_s(S)$. Note that, if $\kappa(S) =1$, then there is a unique elliptic fibration structure on $S$ and hence $\Aut_f(S) = \Aut(S)$; it follows that  $\Aut_f(S)\cap \Aut_s(S)$ is the whole $\Aut_s(S)$.

In any case, we have an exact sequence
\[
1\rightarrow \Aut_B(S) \rightarrow \Aut_f(S) \rightarrow \Aut(B)
\]
where $\Aut_B(S):=\{\sigma\in \Aut_f(S) \mid f\circ \sigma =f\}$. By the decomposition~\eqref{eq: dec}, it is clear that $\sigma\in \Aut_B(S)$ induces the trivial action on $\alpha\in\CH_0(S)_\alb$ if the restriction $\sigma|_F$ to a general fiber $F$ of $f$ is a translation. The latter property is guaranteed if $p_g(S)>0$; see Lemma~\ref{lem: fib trans}. On the other hand, if $p_g(S) = 0$, then $\CH_0(S)_\alb =0$ by \cite{BKL76} and there is nothing to prove. The conclusion is that $\Aut_B(S)\cap \Aut_s(S)$ acts trivially on $\CH_0(S)_\alb$ (Proposition~\ref{prop: fib action}).

Note that $\Aut_B(S)\cap \Aut_s(S)$ also acts trivially on the Jacobian $j\colon J\rightarrow B$. Using this fact, we can reduce the problem to one for finite order symplectic automorphisms of $J$ that fix the distinguished section. This is carried out in Section~\ref{sec: red to J}.

Replacing the elliptic fibration $f\colon S\rightarrow B$ with its Jacobian $j\colon J\rightarrow B$ and $\sigma\in \Aut_f(S)\cap \Aut_s(S)$ with its induced automorphism $\sigma_J\in \Aut_j(J)\cap \Aut_s(J)$, we can assume that $S$ has a section, and then work only with those fibration-preserving symplectic automorphisms $\sigma$ of  finite order such that $|\sigma|= |\sigma_B|$, where $\sigma_B\in \Aut(B)$ is induced by $\sigma$.

Here we observe that, in case $p_g(S)>0$, the canonical map $\varphi_{S}\colon S\dashrightarrow \PP^{p_g-1}$ factors through $f\colon S\rightarrow B$ by the canonical bundle formula for elliptic fibrations. Since $\sigma$ acts trivially on $H^0(S, K_S)$, $\varphi_{S}$ also factors through the quotient map $\pi\colon B\rightarrow B/\langle\sigma_B\rangle$, where $\sigma_B\in \Aut(B)$ is the automorphism induced by $\sigma$. Therefore, we have a commutative diagram

\begin{center}
\begin{tikzpicture}
\node[name = S] at (0,1) {$S$};
\node[name = B] at  (0,0)  {$B$};
 \node[name = BQ] at  (2,0)  {$B/\langle \sigma_B\rangle$};
  \node[name = P] at (4,0) {$\PP^{p_g-1}$};
  \draw[->] (S)--node[left]{{\tiny$f$}}  (B);
 \draw[->] (B)--node[above]{{\tiny$\pi$}} (BQ);
 \draw[->](BQ)--(P);
 \draw[->](S) to [out=15, in=120] node[above]{{\tiny$\varphi_S$}} (P);
 \draw[->] (B) to [out=30, in=150]node[above]{{\tiny$\varphi_B$}}  (P);
\end{tikzpicture}
\end{center}

where  $\varphi_{B}$ is the morphism associated to the linear system $|K_B+L|$ with $L$ is a line bundle of degree $\chi(\mo_S)$ on $B$. If $p_g(S)\geq 2$, we obtain $|\sigma_B| \leq \deg \varphi_{B}$.

If $\chi(\mo_S)\geq 3$, then 
$\deg(K_B+L)=2g(B)-2 +\chi(\mo_S) \geq 2g(B)+1$ and thus $\varphi_{B}$ is an embedding. Therefore, $\sigma_B=\id_B$, and  it follows that $\sigma$ is also the identity (Theorem~\ref{thm: chi geq 3}).

In case $\chi(\mo_S)\in \{1,2\}$, we cannot conclude that $\sigma_B=\id_B$, but a similar consideration yields strong restrictions on what $\sigma_B$ can be. In fact, if we impose the additional condition that $\sigma$ acts trivially also on $H^1(S, \mo_S)$, then $\sigma_B=\id_B$ holds, which implies Theorem~\ref{thm: main2} in this case. In most cases, we can show that $|\sigma| = |\sigma_B|\in\{1,2,3,4,6\}$; using the method of ``enough automorphisms" \cite{IM79}, either $\sigma^2$ or $\sigma^3$ has order at most $2$, and thus acts trivially on $\CH_0(S)_{\alb}$; see Lemma~\ref{lem: inv}.

In case $\chi(\mo_S) = 0$, the relatively minimal model of the surface $S$ over $B$ is an elliptic quasi-bundle and hence is the quotient of a product of two curves. We can thus draw on the finite dimensionality of the Chow motives of such surfaces \cite{Kim05, KMP07}. 

Our arguments have a byproduct concerning (possibly infinite order) symplectic automorphisms of elliptic K3 surfaces:
\begin{thm}[= Corollary~\ref{cor: K3}]
	Let $f\colon S\rightarrow B$ be an elliptic K3 surface. Then $\Aut_f(S)\cap \Aut_s(S)$ acts trivially on $CH_0(S)_\alb$.
\end{thm}

\medskip

\noindent\textbf{Acknowledgements}.
We would like to thank Qizheng Yin, Renjie L\"u, Zhiyuan Li and Xun Yu for their interest in our project and for helpful discussions. This work was supported by the NSFC (No.~11971399 and No.~11771294).

\section{Fibration-preserving automorphisms}
Let $S$ be a smooth projective surface and $f\colon S\rightarrow B$ a fibration, that is, a morphism onto the smooth projective curve $B$ with connected fibers. We define the following subgroups of $\Aut(S)$:
\begin{itemize}[leftmargin=*]
	\item the subgroup of \emph{fibration-preserving automorphisms}
	\[
	\Aut_f(S):=\{\sigma\in \Aut(S) \mid \exists\, \sigma_B\in \Aut(B)  \text{ such that } \sigma_B\circ f = f\circ \sigma\}
	\]
	These automorphisms may permute the fibers of $f$. 
	\item the subgroup of \emph{fiber-preserving automorphisms}
	\[
	\Aut_B(S):=\{\sigma\in \Aut_f(S) \mid f = f\circ \sigma\}
	\]
	These automorphisms preserve each fiber of $f$. 
	
\end{itemize}

There is an obvious exact sequence of groups:
\begin{equation}\label{eq: fib}
	1\rightarrow \Aut_B(S) \rightarrow \Aut_f(S) \xrightarrow{r} \Aut(B)
\end{equation}
where $r$ sends an automorphism $\sigma\in\Aut_f(S)$ to $\sigma_B\in \Aut(B)$ such that $\sigma_B\circ f = f\circ \sigma$.
\begin{lem}\label{lem: finite}
	Let $S$ be a smooth projective surface and $f\colon S\rightarrow B$ a fibration. Suppose one of the following conditions holds:
	\begin{enumerate}[leftmargin=*]
		\item[(1)] $f$ is not isotrivial, that is, not all smooth fibers of $f$ are isomorphic to each other;
		\item[(2)] $g(B)\geq 2$; 
		\item[(3)] $f$ has at least three singular fibers (resp.~one singular fiber) if $g(B)=0$ (resp.~$g(B)=1$).
	\end{enumerate}
	Then the image $\im(r)$ of the homomorphism $r$ in \eqref{eq: fib} is finite.
\end{lem}
\begin{proof}
	(1) If $f$ is not isotrivial then the rational map from $\lambda\colon B\dashrightarrow M_g$ to the moduli space of curves of genus $g$ is generically finite, where $g\geq 1$ is the genus of the general fibers of $f$. On the other hand, for any $b\in B$, the fibers over the points in $\mathrm{Orbit}(b):=\{\sigma_B(b) \mid \sigma_B\in \im(r)\}$ are isomorphic, so $\mathrm{Orbit}(b)$ are mapped to the same point by $\lambda$. Thus $\mathrm{Orbit}(b)$ is finite for all $b$, and it follows that $\im(r)$ is finite.
	
	(2) If $g(B)\geq 2$ then $\Aut(B)$ is finite, so $\im(r)\subset\Aut(B)$ is automatically finite. 
	
	(3) Let $\Sigma=\{b\in B\mid f^*b \text{ is singular}\}$, which is a finite set. Then there is a natural homomorphism $\im(r)\rightarrow \mathrm{Perm(\Sigma)}$ into the (finite) permutation group of $\Sigma$; the kernel of this homomorphism is finite if either $g(B)=0$ and $\#\Sigma>2$ or $g(B)=1$ and $\#\Sigma>0$. 
	
	The proof of the lemma is completed.
\end{proof}

We recall some facts about isotrivial fibrations (\cite[Section~1]{Ser96}). Let $S$ be a smooth projective surface and $f\colon S\rightarrow B$ an isotrivial fibration whose smooth fibers are isomorphic to a fixed curve $F$ with $g(F)\geq 1$.  Birationally $f$ becomes a trivial fibration after a base change. More precisely, there exist a smooth projective curve $\tilde B$, a finite group $G$ acting on $\tilde B$ and $F$, such that $B\cong \tilde B/G$, and the following diagram is commutative:
\begin{equation}\label{diag: isotrivial}
	\begin{tikzpicture}[baseline=(current bounding box.center)]
\node[name = S] at (0,1) {$S$};
\node[name = B] at  (0,0)  {$B$};
\node[name = SQ] at (2,1) {$(\tilde B\times F)/G$};
\node[name = BQ] at (2,0) {$\tilde B/G$};
\draw[->, dashed](S) --(SQ);
\draw[->] (S) -- node[left]{{\tiny $f$}} (B);
\draw[->] (B) -- node[above]{$\tiny\cong$} (BQ);
\draw[->] (SQ) -- node[right]{{\tiny $\pi$}} (BQ);
	\end{tikzpicture}
\end{equation}
where the horizontal dashed arrow is a birational map, $G$ acts diagonally on $\tilde B\times F$, and  $\pi$ is induced by the projection $\tilde B\times F\rightarrow \tilde B$. It is easy to check that, if $b\in B$ is a branch point of the quotient map $\tilde B\rightarrow \tilde B/G\cong B$, then the fiber $f^*b$ of $f$ over $b$ is singular.

\begin{lem}\label{lem: infinite}
	Let $S$ be a smooth projective surface with Kodaira dimension $\kappa(S)\geq 0$, and $f\colon S\rightarrow B$ a fibration. If the image $\im(r)$ of the homomorphism in \eqref{eq: fib} is an infinite group, then $g(B)=1$ and $f$ is a fiber bundle.
\end{lem}
\begin{proof}
	Suppose that $\im(r)$ is infinite.  By Lemma~\ref{lem: finite}, $f$ is isotrivial and $g(B)\leq 1$. Moreover, if $g(B)=1$ then $f$ has no singular fibers, so it is necessarily a fiber bundle. 
	
	Now suppose that $g(B)=0$. We will draw a contradiction by showing that $\im(r)$ is finite in this case, and thus complete the proof. Since $f$ is isotrivial, we have a commutative diagram as in \eqref{diag: isotrivial}. Since $\kappa(\tilde B\times F) \geq \kappa(S)\geq 0$, one has $g(\tilde B)\geq 1$. By the Riemann--Hurwitz formula, the quotient map $\tilde B\rightarrow \tilde B/G\cong B$ has at least three branch points. The fibers of $f$ over these branch points are necessarily singular, so $\im(r)$ is finite by Lemma~\ref{lem: finite}, (3). 
\end{proof}

\begin{cor}\label{cor: finite}
	Let $S$ be a smooth projective surface with Kodaira dimension $\kappa(S)\in\{0,1\}$, and $f\colon S\rightarrow B$ an elliptic fibration. Then the image $\im(r)$ of the homomorphism in \eqref{eq: fib} is finite unless possibly if $S$ is an abelian surface or a bi-elliptic surface.
\end{cor}

\begin{proof}
	If $\im(r)$ is infinite then $f$ is an elliptic bundle over an elliptic curve by Lemma~\ref{lem: infinite}. According to  \cite[Section V.5]{BHPV04}, this can only happen when $S$ is an abelian surface or a bi-elliptic surface.
\end{proof}

Restricting \eqref{eq: fib} to the group of symplectic automorphisms, we obtain another exact sequence
\begin{equation}\label{eq: fib sym}
	1\rightarrow \Aut_B(S)\cap \Aut_s(S) \rightarrow \Aut_f(S)\cap \Aut_s(S) \rightarrow \Aut(B)
\end{equation}
Suppose that $S$ is a surface with Kodaira dimension $\kappa(S)=1$. Then the Iitaka fibration $f\colon S\rightarrow B$, defined by the pluri-canonical systems, is the unique elliptic fibration on $S$. Therefore, any automorphism of $S$ preserves $f$, that is, $\Aut(S)=\Aut_f(S)$. In this case, the exact sequences \eqref{eq: fib} and \eqref{eq: fib sym} can be rewritten as
\begin{equation}\label{eq: fib k1}
	1\rightarrow \Aut_B(S) \rightarrow \Aut(S) \rightarrow \Aut(B)
\end{equation}
and 
\begin{equation}\label{eq: fib sym k1}
	1\rightarrow \Aut_B(S)\cap \Aut_s(S) \rightarrow \Aut_s(S) \rightarrow \Aut(B)
\end{equation}
The image of $\Aut(S) \rightarrow \Aut(B)$ is finite by Corollary~\ref{cor: finite} (see also \cite[ Proposition~1.2]{PS20}).

\medskip

The following fact about elliptic fibrations will be used later on.
\begin{lem}\label{lem: q}
	Let $f\colon S\rightarrow B$ be a relatively minimal elliptic fibration. Then one has
	\[
	q(S)\leq g(B)+1,
	\]
	and equality holds if and only if there exist a smooth projective curve $\tilde B$, an elliptic curve $F$, and a finite group $G$ acting on $\tilde B$ and $F$ such that $S\cong (\tilde B\times F)/G$, $B\cong \tilde B/G$, the action of $G$ on $F$ is by translations, and the diagram 
	\begin{equation}\label{eq: quasi bundle}
	\begin{tikzpicture}[baseline=(current bounding box.center)]
\node[name = S] at (0,1) {$S$};
\node[name = B] at  (0,0)  {$B$};
\node[name = SQ] at (2,1) {$(\tilde B\times F)/G$};
\node[name = BQ] at (2,0) {$\tilde B/G$};
\draw[->](S) --node[above]{{\tiny $\cong$}}(SQ);
\draw[->] (S) -- node[left]{{\tiny $f$}} (B);
\draw[->] (B) -- node[above]{$\tiny\cong$} (BQ);
\draw[->] (SQ) -- (BQ);
	\end{tikzpicture}
	\end{equation}
	commutes, where the right vertical arrow is the natural projection. 
\end{lem}
\begin{proof}
	By \cite[Lemma~1.6]{Ser91} (see also \cite[Lemme]{Bea82}),  one has the inequality $q(S)\leq g(B)+1$. Moreover, in the equality case, one has $\chi(\mo_S)=0$ and thus $f$ is an elliptic quasi-bundle, that is, the possible singular fibers of $f$ are multiples of smooth elliptic curves \cite[Lemma~1.5]{Ser91}. Then the existence of the commutative diagram \eqref{eq: quasi bundle} follows (\cite{Ser96}). We remark that the action of $G$ on $F$ is necessarily by translations, since $g(B) +1= q(S) = g(B) + g(F/G)$ and hence $g(F/G)=1=g(F)$.
\end{proof}

\medskip

Now we state the main result of this section.
\begin{prop}\label{prop: fib action}
	Let $f\colon S\rightarrow B$ be an elliptic fibration. Then the group $\Aut_B(S)\cap \Aut_s(S)$ of fiber-preserving symplectic automorphisms acts trivially on $\CH_0(S)_\alb$.
\end{prop}

We need some preparations before giving the proof of Proposition~\ref{prop: fib action} at the end of this section. 

First, define
\[
\CH_0(S)_f:=\ker (\CH_0(S)\xrightarrow{f_*} \CH_0(B))
\]
and call it the \emph{f-kernel} of $\CH_0(S)$. The following elementary observation about $\CH_0(S)_f$ is the basis of further arguments.
\begin{lem}[cf.~\cite{BKL76}]\label{lem: fib cycle}
	For any $\alpha \in \CH_0(S)_f$, there is a positive integer $d$, such that
	\[
	d\alpha = \sum_i \alpha_i \in \CH_0(S)
	\]
	where $\deg \alpha_i=0$ and $\supp(\alpha_i)$ is contained in a single smooth fiber of $f$ for each $i$.
\end{lem}
\begin{proof}
	By \cite[Fact~3.3]{Voi12}, we can assume that $\supp(\alpha)$ is contained in some union of smooth fibers of $f$. Take an ample smooth curve $C\subset S$, and denote by $d$ the degree of $f|_C\colon C\rightarrow B$. Write $\alpha = \sum_i n_i [p_i]$. Since $\alpha\in \ker(f_*)$, we have
	\[
	f_*\alpha =  \sum_i n_i [f(p_i)]  = 0\in \CH_0(B).
	\]
	It follows that 
	\[
	\sum_i n_i F_{p_i} = f^*\left(\sum_i n_i [f(p_i)]\right) = 0 \in \Pic(S),
	\]
	and hence 
	\[
	\left(\sum_i n_i F_{p_i}\right)\cdot C  = 0 \in \CH_0(S),
	\]
	where $F_{p_i}$ denotes the fiber containing $p_i$.
	Now we can write 
	\begin{align*}
		d\alpha & = \sum_i n_id[p_i] - \left(\sum_i n_i F_{p_i}\right)\cdot C  \\
		&=  \sum_i n_i (d[p_i] - [F_{p_i}\cdot C])
	\end{align*}
	Taking $\alpha_i = n_i (d[p_i] - [F_{p_i}\cdot C])$, we have $\deg \alpha_i=0$ and $\supp(\alpha_i)\subset F_{p_i}$. The proof of the lemma is completed.
\end{proof}

By the universal property of Albanese maps, we have a commutative diagram
\begin{equation}\label{eq: fib alb}
\begin{tikzpicture}[baseline=(current bounding box.center)]
\node[name = S] at (0,1.5) {$\CH_0(S)_\hhom$};
\node[name = AS] at  (0,0)  {$\Alb(S)$};
\node[name = B] at (3,1.5) {$\CH_0(B)_\hhom$};
\node[name = AB] at (3,0) {$\Alb(B)$};
\draw[->](S) --node[above]{{\tiny $f_*$}}(B);
\draw[->] (S) -- node[left]{{\tiny $\alb_S$}} (AS);
\draw[->] (B) -- node[right]{{\tiny $\alb_B$}} (AB);
\draw[->] (AS) -- (AB);
	\end{tikzpicture}
	\end{equation}
By the Abel--Jacobi theorem,  $\alb_B\colon \CH_0(B)_\hhom \rightarrow \Alb(B)$ is an isomorphism. It follows from \eqref{eq: fib alb} that 
\[
\CH_0(S)_\alb = \ker(\alb_S) \subset \ker(f_*) = \CH_0(S)_f,
\]
and $\CH_0(S)_\alb =  \CH_0(S)_f$ if and only if the induced map $\Alb(S)\rightarrow \Alb(B)$ is an isomorphism.

We recall how the translations of a smooth elliptic fiber $F$ of $f$ act on its cycle classes and holomorphic one-forms. The universal cover of the (complex) elliptic curve $F$ is $\CC$, and $F\cong \CC/\Gamma$ for some lattice $\Gamma\subset\CC$. Any $\bar c\in F$ determines an automorphism $\tau_{\bar c}\colon \bar z\mapsto \bar z+\bar c$ of $F$, called the \emph{translation} by $\bar c$; it is descended from the usual translation $\tau_c\colon z\mapsto z+c$ on $\CC$. 

Note that the translation $\tau_{\bar c}$ induces the trivial action on $H^0(F, K_F)$. In fact, a basis element $\xi$ of the one-dimensional vector space $H^0(F, K_F)$ is descended from the one form $dz$ on the universal cover $\CC$. Since $\tau_c^*dz = d(z+c)=dz$ on $\CC$, we have also $\tau_{\bar c}^*\xi = \xi$ on $F$.

Fixing a point $\bar e\in F$ as the origin, one has the identifications 
\[
F\cong \Pic^0(F) \cong \CH_0(F)_{\hom}
\]
by sending $\bar z\in F$ to $\mo_F([\bar z] -[\bar e])\in \Pic^0(F)$ and to $[\bar z] -[\bar e] \in \CH_0(F)_{\hom}$ respectively. In $\CH_0(F)_{\hom}$, we have by the Abel--Jacobi theorem
\[
\tau_{\bar c*}([\bar z] -[\bar e]) = [\bar z + \bar c] -[\bar e + \bar c] = [\bar z] -[\bar e],
\]
so $\tau_{\bar c}$ induces the trivial action on $\CH_0(F)_{\hom}$.

\begin{lem}\label{lem: trans id f kernel}
	Let $f\colon S\rightarrow B$ be an elliptic fibration. Then for any $\sigma\in\Aut_B(S)$ such that its restriction $\sigma_F$ to a general fiber $F$ is a translation, the induced homomorphism $\sigma_*\colon \CH_0(S)_{f,\QQ} \rightarrow  \CH_0(S)_{f,\QQ}$ is the identity. 
\end{lem}
\begin{proof}
	Take any $\alpha\in \CH(S)_{f}$. By Lemma~\ref{lem: fib cycle}, we can write $d\alpha = \sum \alpha_i \in \CH_0(S)$, where $d$ is a positive integer, and for each $i$, $\deg \alpha_i=0$ and $\supp(\alpha_i)$ is contained in a single smooth fiber, say $F_{i}$. Since $\sigma_{F_i}$ is a translation of $F_i$ and $\deg \alpha_i=0$, we have $\sigma_{F_i*}(\alpha_i) = \alpha_i$, viewed as elements in $\CH_0(F_i)$. Pushing the equality to $S$ by the inclusion map $F_i\hookrightarrow S$, we obtain $\sigma_*(\alpha_i) =\alpha_i\in \CH_0(S)$. Therefore, we have 
	\[
	\sigma_*(d\alpha) = \sum_i \sigma_*(\alpha_i) =  \sum_i \alpha_i = d\alpha.
	\]
\end{proof}

\begin{cor}\label{cor: trans id alb kernel}
	Let $f\colon S\rightarrow B$ be an elliptic fibration. Then for any $\sigma\in\Aut_B(S)$ that induces a translation on a general fiber $F$, its action on the Albanese kernel $\sigma_*\colon \CH_0(S)_\alb \rightarrow  \CH_0(S)_\alb$ is the identity. 
\end{cor}
\begin{proof}
	By Lemma~\ref{lem: trans id f kernel}, $\sigma_*\colon \CH_0(S)_{\alb,\QQ} \rightarrow  \CH_0(S)_{\alb, \QQ}$ is the identity. Since  $\CH(S)_\alb$ is torsion free by \cite{Roj80}, we infer that $\sigma_*$ is the identity on $\CH_0(S)_\alb$.
\end{proof}

Let $f\colon S\rightarrow B$ be a relatively minimal elliptic fibration. In order to describe the fiber-preserving symplectic automorphisms of $S$, we take a closer look at the canonical bundle formula (see \cite[Chapter V, Theorem 12.1]{BHPV04} and its proof, and ultimately \cite[Theorem~12.1]{Kod63})
\begin{equation}\label{eq: cbf}
	\omega_S = f^*(f_*\omega_{S/B} \otimes \omega_B) \otimes \mo_S(\sum_i (m_i -1)F_i)
\end{equation}
where $\omega_S=\mo_S(K_S)$ and $\omega_B=\mo_B(K_B)$ are the canonical sheaves of $S$ and $B$ respectively, $\omega_{S/B} = \omega_S\otimes f^*\omega_B^{-1}$ is the relative canonical sheaf of $f$, and the $m_iF_i$'s are the multiple fibers of $f$. There is a natural inclusion of invertible sheaves
\[
f^*(f_*\omega_{S/B} \otimes \omega_B) \hookrightarrow \omega_S
\]
which is an isomorphism over $B^0:=\{b\in B \mid f^*b \text{ is smooth}\}$ and which induces an isomorphism of global sections
\begin{equation}\label{eq: K iso}
	f^*\colon H^0(B, f_*\omega_{S/B} \otimes \omega_B) \xrightarrow{\cong} H^0(S, \omega_S).
\end{equation}
Analytically locally around each $b\in B^0$, there is a small coordinate disk $b\in \Delta\cong \{t\in \CC\mid |t|<\epsilon\}$ such that the sections of $f_*\omega_{S/B} \otimes \omega_B$ has the form
\begin{equation*}
	\xi_t \otimes h(t) dt, 
\end{equation*}
where $h(t)$ is a holomorphic function on $\Delta$ and $\xi_t$ is a basis of $H^0(F_t, K_{F_t})$, varying holomorphically in $t\in \Delta$. Pulling it back to $S$, we obtain a description, which is local in $B$, of the global sections $\omega\in H^0(S, \omega_S)$: 
\begin{equation}\label{eq: local 2-form}
	\omega = h(t) \xi_t \wedge dt.
\end{equation}
This is used in the proof of the next lemma.

\begin{lem}\label{lem: fib trans}
	Let $f\colon S\rightarrow B$ be an elliptic fibration with $p_g(S)>0$. Then a fiber-preserving automorphism $\sigma\in\Aut_B(S)$ is symplectic if and only if it induces translations on the smooth fibers.
\end{lem}

\begin{proof}
	Let $\sigma\in\Aut_B(S)$ be a nontrivial fiber-preserving automorphism and $\sigma_{F}:=\sigma|_{F}$ its restriction to a smooth fiber $F$ of $f$. 
	
	(i) Suppose that  $\sigma\in\Aut_B(S)$ is symplectic. We want to show that $\sigma|_{F}$ is a translation, that is, it acts freely on $F$. Suppose on the contrary that $\sigma_{F}$ has a fixed point $p$. Then $\sigma_{F}$ is of finite order, and $F/\langle\sigma_{F}\rangle$ has genus 0 by the Riemann--Hurwitz formula. It follows that $\sigma$ is of finite order, and a resolution $X$ of the quotient surface $S/\langle\sigma\rangle$ is a $\PP^1$-fibration over $B$. But then 
	\[
	H^0(S, K_S) = H^0(S, K_S)^\sigma \cong H^0(X, K_X) =0.
	\]
	This contradicts the assumption that $p_g(S)>0$.
	
	(ii) Now suppose that $\sigma_{F}\in \Aut(F)$ is a translation. Then $\sigma_F^*\xi = \xi$ for $\xi\in H^0(F, K_F)$; see the discussion above Lemma~\ref{lem: trans id f kernel}.
	Taking a local coordinate $t$ around $f(F)$ in $B$, we can write locally $\omega = h(t)\xi_t\wedge dt$ as in \eqref{eq: local 2-form}, where $\xi_t\in H^0(F_t, K_{F_t})$ is a basis element, so 
	\[
	\sigma^*\omega = \id_B^*h(t)\sigma_{F_t}^*\xi_t \wedge\id_B^*dt  = h(t)\xi_t\wedge dt = \omega
	\]
	It follows that $\sigma$ is a symplectic automorphism of $S$.
\end{proof}


\begin{lem}\label{lem: trans id 1-forms}
	Let $f\colon S\rightarrow B$ be an elliptic surface. Suppose that a fiber-preserving automorphism $\sigma\in\Aut_B(S)$ induces translations on smooth fibers. Then it induces trivial action on $H^{1,0}(S) = H^0(S, \Omega_S^1)$.
\end{lem}
\begin{proof}
	Since  $\sigma\in\Aut_B(S)$ descends to a fiber-preserving automorphism of the relatively minimal model of $S$, we can assume without loss of generality that $f$ is already relatively minimal.
	
	By Lemma~\ref{lem: q}, we have $q(S)\leq g(B)+1$. If $q(S)=g(B)$, then $H^0(S, \Omega_S^1) = f^*H^0(B, K_B)$. Since $\sigma$ induces trivial action on $B$, it induces trivial action on $H^0(B, K_B)$ and hence also on $H^0(S, \Omega^1_S)$.
	
	If $q(S) = g(B) +1$, then $S\cong (\tilde B\times F)/G$ as in \eqref{eq: quasi bundle}. Since $\sigma$ induces $\id_B$ on $B$ and a translation on the fibers $F$, one sees that it induces trivial action on $H^0(S, \Omega_S^1)$.
\end{proof}

Finally, we give the proof of Proposition~\ref{prop: fib action}.
\begin{proof}[Proof of Proposition~\ref{prop: fib action}]
	If $p_g(S)=0$ then $\CH_0(S)_\alb = 0$ by \cite{BKL76} and there is nothing to prove. In case $p_g(S)>0$, it suffices to combine Lemma~\ref{lem: fib trans} and Corollary~\ref{cor: trans id alb kernel}.
\end{proof}


\section{Elliptic surfaces with $\chi(\mo_S)\geq 3$}
In this section, we prove the following theorem.
\begin{thm}\label{thm: chi geq 3}
	Let $S$ be a smooth projective surface with $\kappa(S)=1$ and $\chi(\mo_S)\geq 3$. Then $\Aut_s(S)$ acts trivially on $\CH_0(S)_\alb$.
\end{thm}
\begin{proof}
	Since $\kappa(S)=1$, there is a unique elliptic fibration $f\colon S\rightarrow B$, and any automorphism of $S$ preserves the fibration $f$. 
	
	By Proposition~\ref{prop: fib action}, it suffices to show that $\Aut_s(S)=\Aut_B(S)\cap \Aut_s(S)$, that is, any symplectic automorphism of $S$ preserves the fibers of $f$.
	
	By the canonical bundle formula \eqref{eq: cbf}, we have
	\[
	|K_S| = f^* |K_B+L|+ \sum_i (m_i -1)F_i
	\]
	Thus the canonical map $\varphi_S$ of $S$, induced by the linear system $|K_S|$, factors as
	\[
	\varphi_S\colon S \xrightarrow{f} B \xrightarrow{\varphi_{B}} \PP^{p_g-1} ,
	\]
	where $\varphi_{B}$ is the map defined by the linear system $|K_B+L|$. 
	
	Now, since $\chi(\mo_S)\geq 3$, we have
	\[
	\deg(K_B+L) = 2g(B) -2 + \deg L = 2g(B) -2 +\chi(\mo_S) \geq 2g(B)+1.
	\]
	It follows that $K_B+L$ is  very ample and hence $\varphi_{B}$ is an embedding. A symplectic automorphism $\sigma$ induces an automorphism $\sigma_B\in \Aut(B)$ and the identity on $\PP^{p_g-1}$, and they act equivariantly on the respective varieties. Since $\varphi_{B}$ is an embedding, it can only happen that $\sigma_B=\id_B$. In other words, $\sigma$ preserves each fiber of $f$, which is what we wanted to prove.
\end{proof}

\section{Reduction to the Jacobian fibration}\label{sec: red to J}
Given an elliptic fibration $f\colon S\rightarrow B$, a natural idea is to reduce the problem at hand to the Jacobian fibration $j\colon J\rightarrow B$. The following construction has been used by \cite{BKL76} in proving the vanishing of $\CH(S)_\alb$ for surfaces with $p_g(S)=0$ and $\kappa(S)\leq 1$. We will apply it to deal with the fibration-preserving automorphisms of $S$.

For any irreducible curve $C\subset S$, horizontal with respect to $f$, one can define a rational dominant map $\phi_C\colon S\dashrightarrow J$ over $B$ as follows: to a point $p$ on a smooth fiber $F_b$ over $b\in B$, we associate 
\[
\phi_C(p):=d[p] - C|_{F_b} \in j^*b=\Pic^0(F_b),
\]
where $d$ is the degree of the finite morphism $f|_C\colon C\rightarrow B$. It is clear that $\deg\phi_C=d^2.$

\begin{lem}[cf.~{\cite[page 138]{BKL76}} and {\cite[Proof of Theorem~11.10]{Voi03}}]\label{lem: iso CH}
	For a smooth ample curve $C\subset S$, the induced homomorphism 
	\[
	\phi_{C*}\colon \CH_0(S)_{f,\QQ} \rightarrow \CH_0(J)_{j, \QQ}
	\]
	is an isomorphism, which restricts to an isomorphism between the Albanese kernels $\phi_{C*}\colon  \CH_0(S)_{\alb,\QQ} \rightarrow \CH_0(J)_{\alb,\QQ}$. \end{lem}
\begin{proof}
	We define a homomorphism $\lambda\colon \CH_0(J)_{j,\QQ}\rightarrow \CH_0(S)_{f,\QQ}$ as follows: for any $\gamma\in \CH_0(J)_{j, \QQ}$, we can assume that $\supp(\gamma)$ is on a smooth fiber $j^*b$ of $j$ and $\gamma = [\gamma']-[o_b]$, where $o_b$ denotes the origin of $j^*b=\Pic^0(f^*b)$.  Then we set
	\[
	\lambda(\gamma) =\frac{1}{d^2} \left(([p_1']+\cdots+[p_{d}']) - ([p_1]+\cdots+[p_{d}])\right)
	\]
	where $[p_1]+\cdots+[p_{d}] = C|_{f^*b}$ and, for each $1\leq i\leq d$, $p_i'$ is the unique point of $f^*b$ such that $[p_i']-[p_i] = \gamma'\in \Pic^0(f^*b)$. Then it is straightforward to check that $\lambda$ is the inverse of $\phi_{C*}$.
	
	Note that the irregularities of $S$ and $J$ are the same by the following Lemma~\ref{lem: inv J}, and thus $\phi_{C}$ induces an isomorphism $\phi_{C*}\colon \Alb(S)_\QQ \cong \Alb(J)_\QQ$. In view of the following commutative diagram, where the rows are exact, we infer that the left vertical map $\phi_{C*}\colon \CH_0(S)_{\alb,\QQ}\rightarrow \CH_0(J)_{\alb,\QQ}$ is an isomorphism:
	\begin{equation*}
\begin{tikzpicture}[baseline=(current bounding box.center)]
\node[name = O1] at (-2, 1.5) {$0$};
\node[name = O2] at (-2, 0) {$0$};
\node[name = SA] at (0,1.5) {$\CH_0(S)_{\alb,\QQ}$};
\node[name = Sf] at (3,1.5) {$\CH_0(S)_{f,\QQ}$};
\node[name = AS] at (6,1.5) {$\Alb(S)_\QQ$};
\node[name = JA] at  (0,0)  {$\CH_0(J)_{\alb,\QQ}$};
\node[name = Jj] at (3,0) {$\CH_0(J)_{j,\QQ} $};
\node[name = AJ] at (6,0) {$ \Alb(J)_\QQ$};
\draw[->](O1) -- (SA);
\draw[->] (SA) --(Sf);
\draw[->] (Sf) --(AS);
\draw[->](O2) -- (JA);
\draw[->] (JA) -- (Jj);
\draw[->] (Jj) -- (AJ);
\draw[->] (SA) -- node[right]{\tiny $\phi_{C*}$} (JA);
\draw[->] (Sf) -- node[right]{\tiny $\phi_{C*}$} node[left]{\tiny $\cong$} (Jj);
\draw[->] (AS) -- node[right]{\tiny $\phi_{C*}$}node[left]{\tiny $\cong$}  (AJ);
	\end{tikzpicture}
	\end{equation*}
\end{proof} 
Many of the numerical invariants of an elliptic fibration and its Jacobian fibration turn out to be the same. We give a proof of this fact for lack of an adequate reference.
\begin{lem}[cf.~{\cite[Proposition~5.3.6 on page 308, Corollaries 5.3.4 and 5.3.5 on page 310]{CD89}}]\label{lem: inv J}
	The following equalities hold:
	\[
	\chi(\mo_S) = \chi(\mo_J),\, p_g(S) = p_g(J),\, q(S) = q(J).
	\]
	If $f$ is relatively minimal, then for each $b\in B$, $e(f^*b) = e(j^*b)$, where $e(\cdot)$ denotes the Euler characteristic of a topological space.
\end{lem}
\begin{proof}
	Replacing $f\colon S\rightarrow B$ with the relatively minimal elliptic fibration does not change the invariants $\chi(\mo_S)$, $p_g(S)$ and $q(S)$, as well as its Jacobian fibration. Thus we can assume that $f$ is relatively minimal. Then we have $K_S^2=0$, and hence by the Noether formula 
	\begin{equation}\label{eq: noether S}
		12\chi(\mo_S) = e(S).
	\end{equation}
	On the other hand, $j\colon J\rightarrow B$ is a relatively minimal fibration such that for each $b\in B$, one has $e(f^*b) = e(j^*b)$. It follows that 
	\begin{equation}\label{eq: noether J}
		e(J) = e(S) \text{ and }12\chi(\mo_J) = e(J).
	\end{equation}
	Combining \eqref{eq: noether S} and \eqref{eq: noether J}, we obtain $\chi(\mo_S) = \chi(\mo_J)$.
	
	Since $ \chi(\mo_J) = 1-q(J) + p_g(J)$, it remains to show $q(S) = q(J)$. First, we have the easy inequalities
	\[
	g(B) \leq q(J)\leq q(S) \leq g(B)+1,
	\]
	where the second inequality holds because of the existence of dominant maps such as $\phi_C$ from $S$ to $J$, and the last inequality is given by Lemma~\ref{lem: q}.
	
	Thus, if $q(S) = g(B)$, then $q(J)= q(S) =  g(B)$. 
	
	Suppose now $q(S) = g(B)+1$. Then $S\cong (\tilde B\times F)/G$ as in \eqref{eq: quasi bundle}, and it is straightforward to check that $J= B\times F$. Therefore, 
	\[
	q(J) = g(B)+g(F) = g(B)+1= q(S).
	\]
\end{proof}

The induced map $\phi_{C*}$ of zero cycles and holomorphic forms actually does not depend on the choice of the curve $C$, as the following lemma shows.
\begin{lem}\label{lem: CC' same}
	Let $f\colon S\rightarrow B$ be an elliptic surface, and $C$ and $C'$ be two smooth ample curves on $S $. Suppose that $\deg f|_{C'} = \deg f|_C$. Then the rational maps $\phi_{C}$ and $\phi_{C'}$ from $S$ to $J$ induce the same maps between the $0$-th Chow groups and the spaces of holomorphic forms, that is, 
	\begin{enumerate}[leftmargin=*]
		\item $\phi_{C*} = \phi_{C'*}\colon \CH_0(S)_{f,\QQ} \rightarrow \CH_0(J)_{j,\QQ}$, and
		\item $\phi_C^* = \phi_{C'}^*\colon H^{i,0}(J) \rightarrow H^{i,0}(S)$ for any $i$.
	\end{enumerate}
\end{lem}
\begin{proof}
	Let $d$ be the degree $\deg f|_{C'} = \deg f|_C$. Then
	\[
	\phi_{C'}(p)=d[p] - C'|_{F_b}  = d[p] - C|_{F_b} + (C-C')|_{F_b} = \phi_{C}(p)+(C-C')|_{F_b}\in j^*b,
	\]
	so $\phi_{C'}(p)$ and $\phi_{C}(p)$ differ by a translation of $F_b$ by $(C-C')|_{F_b}$. These translations along the fibers glue to an automorphism of $J$ over $B$, which we denote by $\phi_{C-C'}$. In other words, we have a commutative diagram
	\begin{equation}\label{eq: CC'}
	\begin{tikzpicture}[baseline=(current bounding box.center)]
	\node[name = S] at (1.5, 1) {$S$};
	\node[name = J1] at (0, 0) {$J$};
	\node[name = J2] at (3, 0) {$J$};
	\draw[->, dashed] (S)--node[above]{\tiny $\phi_C$}(J1);
	\draw[->, dashed] (S)--node[above]{\tiny $\phi_{C'}$}(J2);
	\draw[->] (J1)--node[above]{\tiny $\phi_{C-C'}$}(J2);
	\end{tikzpicture}
	\end{equation}
	where $\phi_{C-C'}\in \Aut_B(J)$ induces translations on general fibers of $j\colon J\rightarrow B$. 
	
	By Lemmas~\ref{lem: trans id f kernel}, \ref{lem: trans id 1-forms} and \ref{lem: fib trans}, $\phi_{C-C'}$ induces the identity map on $ \CH_0(J)_{j,\QQ}$ as well as on $ H^{i,0}(J)$. In view of \eqref{eq: CC'}, the desired equalities  $\phi_{C*} = \phi_{C'*}$ and $\phi_C^* = \phi_{C'}^*$ follows.
\end{proof}

By the universal property of $J$, any automorphism $\sigma\in \Aut_f(S)$ preserving the fibration $f$ induces an automorphism $\sigma_J\in \Aut_j(J)$ such that they induce the same automorphism $\sigma_B\in \Aut(B)$ on the base curve $B$ and the following diagram is commutative
\begin{equation}\label{diag: ind aut J}
	\begin{tikzpicture}[baseline=(current bounding box.center), scale=.8]
	\node[name=S1] at (0, 1.5) {$S$};
	\node[name=S2] at (0,0) {$S$};
	\node[name=J1] at (2, 1.5){$J$};
	\node[name=J2] at (2, 0){$J$};
	\draw[->](S1)--node[left]{\tiny $\sigma$}(S2);
	\draw[->, dashed](S1)--node[above]{\tiny $\phi_C$}(J1);
	\draw[->](J1)--node[right]{\tiny $\sigma_J$}(J2);
	\draw[->, dashed](S2)--node[above]{\tiny $\phi_{\sigma(C)}$}(J2);
	\end{tikzpicture}
\end{equation}
This defines a group homomorphism $\Aut_f(S) \rightarrow \Aut_j(J)$, $\sigma\mapsto\sigma_J$. 
For a point $b\in B$ such that the fiber $F_b:=f^*b$ is smooth, we have $j^*b = \Pic^0(F_b)$, and for any $\alpha\in \Pic^0(F_b)$, 
\[
\sigma_J(\alpha) = (\sigma^{-1})^*(\alpha) \in \Pic^0(\sigma(F_b)) = j^*(\sigma_B(b)).
\]

\begin{lem}\label{lem: red to J}
	Let $f\colon S\rightarrow B$ be an elliptic fibration and let $j\colon J\rightarrow B$ the Jacobian of $f$.  Then an automorphism $\sigma\in \Aut_f(S)$ acts as identity on $\CH_0(S)_{f,\QQ}$ (resp.~$\CH_0(S)_{\alb,\QQ}$, resp.~$H^{2,0}(S)$, resp.~$H^{1,0}(S)$) if and only if so does the induced automorphism $\sigma_J\in\Aut_j(J)$ on $\CH_0(J)_{j,\QQ}$ (resp.~$\CH_0(J)_{\alb,\QQ}$, resp.~$H^{2,0}(J)$, resp.~$H^{1,0}(J)$).
\end{lem}
\begin{proof}
	It follows from Lemma~\ref{lem: CC' same} that $\phi_C$ and $\phi_{\sigma(C)}$ in \eqref{diag: ind aut J} induces the same maps on the Chow groups as well as on the spaces $H^{i,0}$, and they are all isomorphisms by Lemmas~\ref{lem: iso CH} and \ref{lem: inv J}. In view of \eqref{diag: ind aut J}, the assertion of the lemma follows.
\end{proof}

\begin{cor}\label{cor: red to J}
	Let $f\colon S\rightarrow B$ be an elliptic fibration and let $j\colon J\rightarrow B$ the Jacobian of $f$. Then $\sigma\in \Aut_f(S)$ acts trivially on $\CH_0(S)_{\alb}$ if 
	and only if
	$\sigma_J\in \Aut_j(J)$ acts trivially on $\CH_0(J)_\alb$.
\end{cor}
\begin{proof}
	The natural maps $\CH_0(S)_\alb\rightarrow \CH_0(S)_{\alb,\QQ}$ and  $\CH_0(J)_\alb\rightarrow \CH_0(J)_{\alb,\QQ}$ are injective by \cite{Roj80}, and hence the corollary follows from Lemma~\ref{lem: red to J}.
\end{proof}

The following lemma on the orders of the induced automorphisms $\sigma_J$ and $\sigma_B$ will be used in Section~\ref{sec: chi leq 2}.
\begin{lem}\label{lem: same order}
	Let $f\colon S\rightarrow B$ be an elliptic fibration such that $p_g(S)>0$, and let $j\colon J\rightarrow B$ the Jacobian of $f$.  Let $\sigma\in\Aut_f(S)\cap\Aut_s(S)$ be a symplectic fibration-preserving automorphism. Then the induced automorphisms $\sigma_J\in\Aut_j(J)$ and $\sigma_B\in\Aut(B)$ have the same order.
	\begin{proof}
		For any integer $n$, we have
		\begin{equation}\label{eq: order J B}
			j\circ \sigma_J^n = \sigma_B^n\circ j
		\end{equation}
		If the order $|\sigma_B|$ is infinite, then $\sigma_J$ necessarily has infinite order.
		
		Now suppose that $\sigma_B$ has finite order $m$. Then $\sigma^m$ lies in $\Aut_B(S)\cap \Aut_s(S)$. By Lemma~\ref{lem: fib trans}, $\sigma^m$ induces translations on the smooth fibers $F_b$. It follows that $\sigma_J^m$ induces the identity on $j^*b=\Pic^0(F_b)$ and is itself the identity. Therefore, $|\sigma_J|$ is finite, with order dividing $m=|\sigma_B|$. On the other hand, $m$ divides $|\sigma_J|$ by \eqref{eq: order J B}. We infer that $|\sigma_J| = |\sigma_B|$.
	\end{proof}
\end{lem}

\section{Elliptic surfaces with $\chi(\mo_S)\leq 2$}\label{sec: chi leq 2}
In this section we deal with elliptic surfaces $f\colon S\rightarrow B$ with $\chi(\mo_S)\leq 2$.

We need a lemma for the action of the Klein group $(\ZZ/2\ZZ)^2$ on $\CH_0(S)$. It is based on the idea of  ``enough automorphisms" of \cite{IM79}.
\begin{lem}\label{lem: Klein}
	Let $S$ be a smooth projective surface. Let $G=\langle \sigma, \tau \rangle\cong (\ZZ/2\ZZ)^2$ be a subgroup of $\Aut(S)$ such that the smooth models of the quotient surfaces $S/\langle \tau\rangle$ and $S/\langle \sigma\tau\rangle$ are not of general type and have vanishing geometric genera. Then $\sigma$ induces the identity on the $\CH_0(S)_{\alb}$.
\end{lem}
\begin{proof}
	Let $X_1$ and $X_2$ be the smooth models of  $S/\langle \tau\rangle$ and $S/\langle \sigma\tau\rangle$ respectively. Then, by \cite{BKL76}, we have
	\[
	\CH_0(S)_{\alb, \QQ}^\tau = \CH_0(X_1)_{\alb,\QQ} = 0,\, \CH_0(S)_{\alb, \QQ}^{\sigma\tau} = \CH_0(X_2)_{\alb,\QQ} = 0 
	\]
	Since $\tau$ and $\sigma\tau$ are both involutions, they act as $-\id$ on $\CH_0(S)_{\alb, \QQ}$. It follows that $\sigma = (\sigma\tau)\tau$ acts trivially on $\CH_0(S)_{\alb, \QQ}$. Since $\CH_0(S)_{\alb}$ has no torsion by \cite{Roj80}, the lemma follows.
\end{proof}

\begin{lem}\label{lem: inv}
	Let $f\colon S\rightarrow B$ be a smooth projective elliptic surface. Then any symplectic involution of $S$ acts trivially on $\CH_0(S)_{\alb}$.
\end{lem}
\begin{proof}
	If $\kappa(S)\leq 0$, this is a consequence of \cite{BKL76} and \cite{Voi12}. 
	
	In the following we assume that $\kappa(S)=1$. Let $\sigma\in \Aut_s(S)$ be a symplectic involution. By Lemma~\ref{lem: red to J}, it suffices to prove that  the induced automorphism $\sigma_J\in \Aut_s(J)$ acts trivially on $\CH_0(J)_\alb$, where $j\colon J\rightarrow B$ is the Jacobian fibration of $f$. Note that $|\sigma_J|\leq |\sigma|= 2$. If $\sigma_J=\id_J$, then there is nothing to prove. We can thus assume that $\sigma_J$ is also an involution.
	
	Note that the $o$-section of $j$ is preserved by $\sigma_J$. Let $\tau\in \Aut_B(J)$ be the involution that restricts to $-\id_F$ on a general fiber $F$ of $j$. Then the subgroup $G=\langle \sigma_J, \tau \rangle<\Aut_j(J)$, generated by $\sigma_J$ and $\tau$, is isomorphic to $(\ZZ/2\ZZ)^2$. It is easy to see that the smooth models of the quotient surfaces $J/\langle\tau\rangle$ and $J/\langle \tau\sigma_J\rangle$ have vanishing geometric genera. By Lemma~\ref{lem: Klein}, $\sigma_{J}$ acts as the identity on $\CH_0(J)_\alb$.
\end{proof}

\begin{thm}\label{thm: chi eq 2}
	Let $f\colon S\rightarrow B$ be an elliptic fibration with $\chi(\mo_S)=2$. Then $\Aut_f(S)\cap \Aut_s(S)$ acts trivially on $\CH_0(S)_{f,\QQ}$. 
	
	As a consequence, $\Aut_f(S)\cap \Aut_s(S)$ acts trivially on $\CH_0(S)_{\alb}$.
\end{thm}
\begin{proof}
	Since $\chi(\mathcal{O}_S)=2$, one has $p_g(S)=q(S)+1\geq 1$. Let $j\colon J\to B$ be the Jacobian fibration of $f$. By Lemma \ref{lem: red to J}, it suffices to show that $\sigma_J$ acts trivially on $\CH_0(J)_{j,\QQ}$, where $\sigma_J\in \Aut_j(J)\cap \Aut_s(J)$ is the automorphism induced by $\sigma$. 
	
	First we assume that $q(J)=0$. In this case $J$ is an elliptic $K3$ surface. Recall that the induced automorphisms $\sigma_J\in\Aut(J)$ and $\sigma_B\in\Aut(B)$ have the same order by Lemma~\ref{lem: same order}, which is finite by Corollary~\ref{cor: finite}.
	
	Now the triviality of the action of $\sigma_J$ on $\CH_0(J)_{j,\QQ}$ follows from the results of Voisin \cite{Voi12} and Huybrechts \cite{Huy12}.
	
	Now we can assume that $q(J)>0$. By the canonical bundle formula, we have $|K_J|= j^* |K_B+L|$, where $L$ is a line bundle of degree $\chi(\mo_J)=\chi(\mo_S)=2$.  Since $\deg(K_B+L)=2g(B)$, the linear system $|K_{B}+L| $ is base point free and the map $\varphi_{B}$ defined by $|K_{B}+L|$ is a morphism. It follows that the canonical map $\varphi_{J}$ of $J$ is a morphism which factors as
	\[
	\varphi_{J} \colon J \xrightarrow{j} B \xrightarrow{\varphi_{B}} \PP^{p_g-1}
	\]
	where $p_g:=p_g(J) = p_g(S)$.
	
	On the one hand,  $\deg(K_B+L)=2g(B)=2(p_g-1)$  by Riemann--Roch. On the other hand,  we know that $\deg(K_B+L)=\deg(\varphi_{B})\cdot \deg(\im(\varphi_{B}))$ and $\deg(\im(\varphi_{B}))\geq p_g-1$. So $\deg(\varphi_{B})\leq 2$. 
	
	Since $\sigma_J$ acts trivially on $H^0(J, K_J)$, the morphism $\varphi_{B}$ factors through the quotient map $B\rightarrow B/\sigma_B$. Therefore, one has
	\[
	|\sigma_J| = |\sigma_B| \leq \deg\varphi_{B} \leq 2.
	\]
	We are done by Lemma~\ref{lem: inv}.
\end{proof}
We have the following two immediate corollaries.
\begin{cor}
	Let $S$ be a smooth projective surface with $\kappa(S)=1$ and $\chi(\mathcal{O}_S)=2$. Then $\Aut_s(S)$ acts trivially on $\CH_0(S)_{\alb}$.
\end{cor}
\begin{cor}\label{cor: K3}
	Let $f: S\to B$ be an elliptic $K3$ surface. Then $\Aut_f(S)\cap \Aut_s(S)$ acts trivially on $\CH_0(S)_{\alb}$.
\end{cor}

Next we treat the case where $\chi(\mathcal{O}_{S})=1$.
\begin{thm}\label{thm: chi eq 1}
	Let $f:S\to B$ be a smooth projective elliptic surface with $\chi(\mo_S)=1$. Then $\Aut_s(S)\cap\Aut_f(S)$ acts trivially on $\CH_0(S)_{\alb}$ if $p_g(S)=q(S)\notin\{1,2\}$. Otherwise, the image of  the homomorphism $\Aut_s(S)\cap \Aut_f(S)\to \Aut(\CH_0(S)_{\alb})$ has order at most $3$.
\end{thm}
\begin{proof}
	If $p_g(S)=q(S)=0$, then $\CH_0(J)_{\alb}=0$ by \cite{BKL76}, and there is nothing to prove. 
	
	So we may assume that $p_g(S)=q(S)>0$.
	Let $j\colon J\rightarrow B$ be the Jacobian fibration of $f$. Let $\sigma_J\in \Aut_j(J)\cap\Aut_s(J)$ and $\sigma_B\in\Aut(B)$ be the automorphisms induced by $\sigma$. By Corollary~\ref{cor: red to J}, it suffices to show that $\sigma_J$ acts trivially on $\CH_0(J)_{\alb}$. 
	
	Since $\chi(\mo_S)=1$, the surface cannot be abelian or bielliptic, so $\sigma_B$ is of finite order by Corollary~\ref{cor: finite}. Also, $\sigma_J$ has the same finite order as $\sigma_B$ by Lemma~\ref{lem: same order}.
	
	By Lemma~\ref{lem: inv J}, $\chi(\mo_J)=\chi(\mathcal{O}_{S})=1$ and thus $q(J) = p_g(J)$. Since $\chi(\mo_J)>0$, the fibration $j\colon J\rightarrow B$ cannot be an elliptic quasi-bundle and it follows that $g(B) = q(J)$.
	
	By the canonical bundle formula,  we have $|K_J|=j^{\ast}|K_B+L|$, where $L$ an invertible sheaf on $B$ of degree $\chi(\mo_J)=1$. Then the canonical map of $J$ factors through $j$ followed by the map $\varphi_{B}$ induced by the linear system $|K_B+L|$ as in the proof of Theorem~5.3. 
	
	If $g(B)=q(J)\geq 3$, then 
	\[
	|\sigma_J|=|\sigma_B|\leq \deg (\varphi_{B})\leq 2.
	\]
	In this case, $\sigma_J$ acts trivially on $\CH_0(J)_{\alb}$ by Lemma~\ref{lem: inv}.
	
	If $g(B)=q(J)=2$, then 
	\[
	|\sigma_J|\leq \deg (\varphi_{B})\leq 3.
	\]
	If $g(B)=q(J)=1$, then $\deg(K_B+L)=1$ and $|K_B+L|$ consists of a unique element, say $p\in B$, which is necessarily fixed by $\sigma_B$. 
	In these last two cases, the order $|\sigma_J|$ is one of $\{1,2,3,4,6\}$. It follows that either $\sigma_J^2$ or $\sigma_J^3$ has order at most $2$, and thus acts trivially on $\CH_0(J)_{\alb}$ by Lemma~\ref{lem: inv}. This completes the proof.
\end{proof}

Now we deal with the case where $\chi(\mathcal{O}_{S})=0$.

\begin{thm}\label{chi eq 0}
	Let $S$ be a smooth projective surface with $\chi(\mathcal{O}_S)=0$. Then $\Aut_s(S)$ acts trivially on $\CH_0(S)_{\alb}$.
\end{thm}
\begin{proof}
	Since $\chi(\mathcal{O}_S)=0$, we have $\kappa(S)\leq 1$. If $\kappa(S)\leq 0$ or $p_g(S)=0$, then the assertion follows from \cite{BKL76}.
	
	Therefore, we can assume that $\kappa(S)=1$ and $p_g(S)>0$. We can also assume that $S$ is minimal.  Let $f\colon S\to B$  be the Iitaka fibration of $S$ and $j\colon J\rightarrow B$ its Jacobian. For any $\sigma\in \Aut_s(S)$, the induced automorphism $\sigma_J\in \Aut_s(J)\cap \Aut_j(J)$ has finite order by Lemma~\ref{lem: same order} and Corollary~\ref{cor: finite}. 
	
	By Lemma~\ref{lem: red to J}, it suffices to show that $\sigma_J$ induces the trivial action on $\CH_0(J)_{\alb, \QQ}$.  Since $\chi(\mathcal{O}_J)=\chi(\mo_S) = 0$, $j$ is a quasi-bundle and hence $J$ is isogenous to a product of curves.  It follows that the Chow motive $h(J)$ is finite dimensional in the sense of Kimura, and the assertion follows from Lemma \ref{lem:fdtoSBC}.
\end{proof}
\begin{rmk} 
	It would be interesting to give a direct proof of Theorem~\ref{chi eq 0} without involving the theory of Chow motives.
\end{rmk}

The following lemma should be well-known to experts. We write down a proof for lack of an adequate reference.
\begin{lemma}\label{lem:fdtoSBC}
	Let $S$ be a smooth projective surface with $p_g(S)>0$. Assume that the Chow motive $h(S)$ of $S$ is finite dimensional in the sense of Kimura \cite{Kim05}. Then any symplectic automorphism $\sigma\in \Aut_s(S)$ of finite order acts as the identity on $\CH_0(S)_{\alb}$. 
\end{lemma}
\begin{proof}
	The Chow motive of $S$ has a Chow-K\"unneth decomposition (\cite[Proposition 7.2.1]{KMP07}):
	\[
	h(S) = h_0(S) \oplus  h_1(S) \oplus  h_2(S) \oplus  h_3(S) \oplus  h_4(S)
	\]
	in the category of Chow motives with rational coefficients. There is a further decomposition for $h_2(S)$ (\cite[Proposition 7.2.3]{KMP07}):
	\[
	h_2(S) = h_2^\alg(S) \oplus t_2(S),
	\]
	where $h_2^\alg(S)$ denotes the algebraic part and $t_2(S)$ the transcendental part. We have
	\[
	\CH_0(t_2(S)) = \CH_0(S)_{\alb,\QQ}\text{ and }  H^2(t_2(S))=H^2_{\tr}(S,\QQ),
	\]
	where $H^2_{\tr}(S,\QQ)$ denotes the transcendental part of $H^2(S,\QQ)$.
	
	Since the motive $h(S)$ is finite dimensional in the sense of Kimura \cite{Kim05}, its direct summand $t_2(S)$ is also finite dimensional. For any symplectic automorphism $\sigma\in \Aut_s(S)$, it acts trivially on $H^2_{\tr}(S,\QQ) = H^2(t_2(S))$. Therefore, $(\Gamma_\sigma -\Delta_S)_*\colon t_2(S)\rightarrow t_2(S)$ is a numerically trivial morphism, where $\Gamma_\sigma$ is the graph of $\sigma$ and $\Delta_S\subset S\times S$ is the diagonal. Then, by Kimura's nilpotence theorem \cite[Proposition 7.5]{Kim05}, $(\Gamma_\sigma -\Delta_S)_*$ is nilpotent as an endomorphism of $t_2(S)$. 
	
	It follows that the action of $\sigma$ on $\CH_0(t_2(S)) = \CH_0(S)_{\alb,\QQ}$ is unipotent. Since $\sigma$ is of finite order, we infer that $\sigma$ acts as the identity on $ \CH_0(S)_{\alb,\QQ}$. It acts trivially also on $\CH_0(S)_{\alb}$ because $\CH_0(S)_{\alb}$ has no torsion by \cite{Roj80}.
\end{proof}

Finally, strengthening the hypothesis in Conjecture \ref{conj: sym},  we obtain
\begin{thm}
	Let $S$ be a smooth projective surface with $\kappa(S)=1$. If an automorphism $\sigma\in \Aut(S)$ induces the trivial action on $H^{i,0}(S)$ for $i>0$, then it induces the trivial action on $\CH_0(S)_\alb$.
\end{thm}
\begin{proof}
	Let $f:S\to B$ be the Iitaka fibration of $S$ and $\sigma$ an automorphism of $S$ acting trivially on $H^{i,0}(S)$ for $i>0$. By Theorems~\ref{thm: chi geq 3}, \ref{thm: chi eq 2}, \ref{thm: chi eq 1} and \ref{chi eq 0}, it is enough to show the result for surfaces with $q(S)=p_g(S)\in\{1,2\}$. 
	
	Let  $\sigma_{B}\in \Aut(B)$ be the automorphism induced by $\sigma$. It suffices to show that $\sigma_B=\id_B$, since then $\sigma\in \Aut_B(S)\cap \Aut_s(S)$ and we can conclude by Proposition~\ref{prop: fib action}.
	
	Observe that, since $\sigma$ acts trivially on $H^{1,0}(S)$ and $f^{\ast}\colon H^{1,0}(B)\rightarrow H^{1,0}(S)$ is injective, $\sigma_{B}$ acts trivially on $H^{1,0}(B)$. It follows that $\sigma_B =\id_B$ if $g(B)= 2$. In case $g(B)=1$, the automorphism $\sigma_B$ is necessarily a translation. Since $\chi(\mathcal{O}_S)>0$, by the holomorphic Lefschetz fixed point formula, the fixed locus $S^{\sigma}$ is non-empty. Thus the translation $\sigma_B$ fixes a nonempty subset $f(S^{\sigma})\subset B$, and we infer that $\sigma_{B}=\id_{B}$.
\end{proof}

\end{document}